\documentclass[a4paper,12pt,reqno]{amsart}

\usepackage{amsfonts, color}
\usepackage{amsmath}
\usepackage{amssymb}
\usepackage[a4paper]{geometry}
\usepackage{mathrsfs}
\usepackage[colorlinks]{hyperref}

\usepackage{hyperref}
\renewcommand\eqref[1]{(\ref{#1})}

\makeatletter
\newcommand*{\mint}[1]{%
  \mint@l{#1}{}%
}
\newcommand*{\mint@l}[2]{%
  \@ifnextchar\limits{%
    \mint@l{#1}%
  }{%
    \@ifnextchar\nolimits{%
      \mint@l{#1}%
    }{%
      \@ifnextchar\displaylimits{%
        \mint@l{#1}%
      }{%
        \mint@s{#2}{#1}%
      }%
    }%
  }%
}
\newcommand*{\mint@s}[2]{%
  \@ifnextchar_{%
    \mint@sub{#1}{#2}%
  }{%
    \@ifnextchar^{%
      \mint@sup{#1}{#2}%
    }{%
      \mint@{#1}{#2}{}{}%
    }%
  }%
}
\def\mint@sub#1#2_#3{%
  \@ifnextchar^{%
    \mint@sub@sup{#1}{#2}{#3}%
  }{%
    \mint@{#1}{#2}{#3}{}%
  }%
}
\def\mint@sup#1#2^#3{%
  \@ifnextchar_{%
    \mint@sup@sub{#1}{#2}{#3}%
  }{%
    \mint@{#1}{#2}{}{#3}%
  }%
}
\def\mint@sub@sup#1#2#3^#4{%
  \mint@{#1}{#2}{#3}{#4}%
}
\def\mint@sup@sub#1#2#3_#4{%
  \mint@{#1}{#2}{#4}{#3}%
}
\newcommand*{\mint@}[4]{%
  \mathop{}%
  \mkern-\thinmuskip
  \mathchoice{%
    \mint@@{#1}{#2}{#3}{#4}%
        \displaystyle\textstyle\scriptstyle
  }{%
    \mint@@{#1}{#2}{#3}{#4}%
        \textstyle\scriptstyle\scriptstyle
  }{%
    \mint@@{#1}{#2}{#3}{#4}%
        \scriptstyle\scriptscriptstyle\scriptscriptstyle
  }{%
    \mint@@{#1}{#2}{#3}{#4}%
        \scriptscriptstyle\scriptscriptstyle\scriptscriptstyle
  }%
  \mkern-\thinmuskip
  \int#1%
  \ifx\\#3\\\else_{#3}\fi
  \ifx\\#4\\\else^{#4}\fi
}
\newcommand*{\mint@@}[7]{%
  \begingroup
    \sbox0{$#5\int\m@th$}%
    \sbox2{$#5\int_{}\m@th$}%
    \dimen2=\wd0 %
    \let\mint@limits=#1\relax
    \ifx\mint@limits\relax
      \sbox4{$#5\int_{\kern1sp}^{\kern1sp}\m@th$}%
      \ifdim\wd4>\wd2 %
        \let\mint@limits=\nolimits
      \else
        \let\mint@limits=\limits
      \fi
    \fi
    \ifx\mint@limits\displaylimits
      \ifx#5\displaystyle
        \let\mint@limits=\limits
      \fi
    \fi
    \ifx\mint@limits\limits
      \sbox0{$#7#3\m@th$}%
      \sbox2{$#7#4\m@th$}%
      \ifdim\wd0>\dimen2 %
        \dimen2=\wd0 %
      \fi
      \ifdim\wd2>\dimen2 %
        \dimen2=\wd2 %
      \fi
    \fi
    \rlap{%
      $#5%
        \vcenter{%
          \hbox to\dimen2{%
            \hss
            $#6{#2}\m@th$%
            \hss
          }%
        }%
      $%
    }%
  \endgroup
}

\numberwithin{equation}{section}
\theoremstyle{plain}
\newtheorem{thm}{Theorem}[section]

\newtheorem{prop}[thm]{Proposition}
\newtheorem{cor}[thm]{Corollary}

\theoremstyle{definition}

\theoremstyle{remark}
\newtheorem{rem}{Remark}[section]
\newtheorem{defn}{Definition}

\numberwithin{equation}{section}

\DeclareMathOperator{\Tr}{Tr}
\DeclareMathOperator{\spn}{span}
\let\Re\relax
\DeclareMathOperator{\Re}{Re}
\DeclareMathOperator{\supp}{supp}
\DeclareMathOperator*{\esssup}{ess\,sup}

\begin{document}

\title[Singular Klein-Gordon equation on a bounded domain]{Singular Klein-Gordon equation on a bounded domain}

\author[M. Ruzhansky]{Michael Ruzhansky}
\address{
  Michael Ruzhansky:
  \endgraf
  Department of Mathematics: Analysis, Logic and Discrete Mathematics
  \endgraf
  Ghent University, Krijgslaan 281, Building S8, B 9000 Ghent
  \endgraf
  Belgium
  \endgraf
  and
  \endgraf
  School of Mathematical Sciences
  \endgraf
  Queen Mary University of London
  \endgraf
  United Kingdom
  \endgraf
  {\it E-mail address} {\rm michael.ruzhansky@ugent.be}
}

\author[A. Yeskermessuly ]{Alibek Yeskermessuly}
\address{
  Alibek Yeskermessuly:
  \endgraf   
  Faculty of Natural Sciences and Informatization
  \endgraf
  Altynsarin Arkalyk Pedagogical Institute, Auelbekov, 17, 110300 Arkalyk, Kazakhstan
  \endgraf  
  {\it E-mail address:} {\rm alibek.yeskermessuly@gmail.com}
  }

\thanks{This research was funded by the Science Committee of the Ministry of Science and Higher Eduacation of the Republic of Kazakhstan (Grant No. AP23486342), by the FWO Odysseus 1 grant G.0H94.18N: Analysis and Partial Differential Equations, and by the Methusalem programme of the Ghent University Special Research Fund (BOF) (Grant number 01M01021). MR is also supported by EPSRC grant EP/V005529/1 and FWO Research Grant G083525N}


\begin{abstract}
In this paper, we consider the wave equation for the Laplace operator with potential, initial data, and nonhomogeneous Dirichlet boundary condition. We establish a weak solution by using traces and extension domains. We also establish the existence, uniqueness and consistency of the very weak solution for the wave equation with singularities in the potential, initial data, source term, boundary and boundary condition.
\end{abstract}

\maketitle

\section{Introduction}
Our aim in this paper is to investigate the well-posedness of the wave equation for the Laplace operator with a potential, initial data, and non-homogeneous Dirichlet boundary conditions, when the coefficients exhibit singularities. Specifically, let us consider the equation
\begin{equation}\label{1.1}
    \partial^2_tu(t,x)-\Delta u(t,x)+V(x)u(t,x)=f(t,x),\quad (t,x)\in [0,T]\times \Omega\subset\mathbb{R}^n,
\end{equation}
with initial conditions
\begin{equation}\label{1.2}
    u(0,x)=u_0(x),\quad  \partial_t u(0,x)=u_1(x),\quad x\in \Omega,
\end{equation}
and boundary condition
\begin{equation}\label{1.3}
    u(t,x)=g(t,x),\quad (t,x)\in [0,T]\times\partial\Omega,
\end{equation}
with $V$, $f$, $u_0$, $u_1$, $g$ and $\partial\Omega$ allowing singularities.

It is well known and has been thoroughly studied (see, e.g., \cite{Af-Roz},  \cite{DHP07}, \cite{DRP-22}) when the coefficients are regular functions. It is difficult to establish the well-posedness of the initial/boundary problem \eqref{1.1}-\eqref{1.3}, where the spatial potential $V$, the source term $f$, the initial data $(u_0,u_1)$, and the nonhomogeneous Dirichlet boundary condition $g$ are allowed to be nonregular functions, such as Dirac delta functions, and when the boundary may be irregular. We conducted this study within the framework of very weak solutions. The motivation for adopting this approach stems from the fact that, when the equation involves products of distributional terms, it becomes impossible to formulate the problem in the distributional setting. This issue is related to the well-known result by Schwartz \cite{Sch54}, which demonstrates the impossibility of multiplying distributions.

To effectively address this problem, the concept of very weak solutions was introduced in \cite{GR15} to analyze second-order hyperbolic equations with strongly singular coefficients. Since then, this approach has been further developed for various problems, as seen in works such as \cite{RT17a}, \cite{RT17b}, \cite{MRT19}, \cite{ART19}, \cite{ARST21a}, \cite{ARST21b}, \cite{ARST21c}, \cite{CRT21}, \cite{CRT22a}, \cite{CRT22b}, \cite{SW22}, \cite{CDRT23} and \cite{GS24}, among others. Several of these studies, particularly \cite{ARST21a}, \cite{ARST21b}, \cite{ARST21c}, \cite{CRT21}, \cite{CRT22a}, and \cite{CRT22b}, employ energy methods in their arguments. Recent works \cite{RSY22}, \cite{RY22}, \cite{RY24a} and \cite{RY24b} have explored the existence of solutions to initial/boundary value problems for the Sturm-Liouville operator, considering various types of time-dependent singular coefficients. In these studies, separation of variables techniques (see, e.g. \cite{Evans}) were used to derive explicit formulas for classical solutions. 

The main difference between our work and previous studies is that we also consider the wave equation with singular nonhomogeneous Dirichlet boundary conditions, which has not been explored before, with singular potential and singular boundary.

In works \cite{KL09}, \cite{KL11}, \cite{MeWi} it is shown that the regularity of the initial data enables the derivation of regular solutions up to the boundary. Specifically, a $C^1$-regular boundary allows for the use of the Sobolev space $H^2(\Omega)$ in the Laplace domain \cite{Evans}. However, in the nonconvex case with a Lipschitz boundary \cite{Gris}, it is no longer possible, as we are limited to $H^1$. The convexity condition prevents the formation of inward angles that are responsible for creating singularities.

For equations with non-homogeneous Dirichlet boundary conditions, the methods described in \cite{Af-Roz} and \cite{DRP-22} are appropriate, where well-posedness is achieved through the Galerkin approximations by using traces and extension domains. Additionally, they obtained weak well-posedness of the wave equation with NTA domain in $\mathbb{R}^2$ and admissible domain in $\mathbb{R}^3$ regarding in the work \cite{Nyst} and \cite{Xie} respectively.

In our work, we extend these methods to cases where the coefficients exhibit singularities, employing the concept of very weak solutions.

\section{Traces and extension domains}

To start our research, we need some preliminaries on traces and extension domains. We introduce the existing results about traces and extension domains. For more detailed study, we cite \cite{Af-Roz}, \cite{JW-84}, \cite{DRP-22}.

    Let $\Omega\subset \mathbb{R}^n$ be a domain, $k\in \mathbb{N}$ a non-negative integer, and $1\leq p\leq \infty$. The Sobolev space $W^k_p(\Omega)$ is the space of functions $u\in L^p(\Omega)$ whose weak derivatives $D^\alpha$ of order $|\alpha|\leq k$ exist in the distributional sense and belong to $L^p(\Omega)$. Formally, we define 
    \[
    W^k_p(\Omega)=\left\{u\in L^p(\Omega): D^\alpha u\in L^p(\Omega)\text{ for all } |\alpha|\leq k\right\},
    \]
    where $D^\alpha u$ denotes the weak derivative of $u$ corresponding to the multi-index $\alpha$. The norm on this space is given by
    \[
    \|u\|_{W^k_p(\Omega)}= \left(\sum\limits_{|\alpha|\leq k}\|D^\alpha u\|^p_{L^p(\Omega)}\right)^{\frac{1}{p}}\quad \text{for }1\leq p<\infty,
    \]
    and for $p=\infty$
    \[
    \|u\|_{W^k_\infty(\Omega)}= \max\limits_{|\alpha|\leq k}\|D^\alpha u\|_{L^\infty(\Omega)}.
    \]

\begin{defn}[$W^k_p$-extension domains \cite{JW-84}, \cite{Af-Roz}]\label{def.1}
    A domain $\Omega\subset \mathbb{R}^n$ is called a $W^k_p$-extension domain ($k\in\mathbb{N}$) if there exists a bounded linear extension operator $E: W^k_p(\Omega)\to W^k_p(\mathbb{R}^n)$. This means that for all $u\in W^k_p(\Omega)$ there exists a $v=Eu\in W^k_p(\mathbb{R}^n)$ with $v\big|_\Omega=u$ and we have 
    \[
    \|v\|_{W^k_p(\mathbb{R}^n)}\leq C\|u\|_{W^k_p(\Omega)}
    \]
    with a constant $C>0$.
\end{defn}

\begin{defn}[$d$-measure \cite{JW-84}]\label{d-measure}
    Let $F$ be a closed non-empty subset of $\mathbb{R}^n$ and $d$ a real number satisfying $0<d\leq n$. The closed ball with center $x$ and radius $r$ is denoted by $B(x,r)$. A positive Borel measure $\mu$ with support $F$, $\supp \mu=F$, is called a $d$-measure on $F$ if, for some constants $c_1,\,c_2>0$,
    \[
    c_1r^d\leq \mu(B(x,r))\leq c_2r^d, \,\, \text{for } x\in F,\, 0<r\leq 1.
    \]
\end{defn}

\begin{defn}[Ahlfors $d$-regular set or $d$-set \cite{JW-84}, \cite{Af-Roz}]\label{d-set}
    A closed, non-empty subset of $\mathbb{R}^n$ is a $d$-set $(0<d\leq n)$ if there exists a $d$-measure on $F$.
\end{defn}

\begin{defn}[Markov's local inequality \cite{JW-84}, \cite{Af-Roz}]\label{d.2}
    A closed subset $V$ in $\mathbb{R}^n$ preserves Markov's local inequality if for every fixed $k\in \mathbb{N}^*$, there exists a constant $c=c(V,n,k)>0$, such that 
    \[
    \max\limits_{V\cap\overline{B_{r}(x)}}|\nabla P|\leq \frac{c}{r}\max\limits_{V\cap \overline{B_r(x)}}|P|
    \]
    for all polynomials $P$ of degrees at most $k$ (denoted by $\mathcal{P}_k$) and all closed balls $\overline{B_{r}(x)}$, $x\in V$ and $0<r\leq 1$.
\end{defn}

\begin{defn}[\cite{JW-84}, \cite{Af-Roz}, \cite{DRP-22}]\label{d.C^k_p}
    For a set $\Omega\subset \mathbb{R}^n$ of positive Lebesgue measure, we define
\[
C^k_p(\Omega)=\left\{f\in L^p(\Omega)|f^\sharp_{k,\Omega}(x)=\sup\limits_{r>0} r^{-k} \inf\limits_{P\in \mathcal{P}^{k-1}}\frac{1}{\mu(B_r(x))}\int_{B_r(x)\cap \Omega}|f-P|dy\in L^p(\Omega)\right\}
\]
with norm $\|f\|_{C^k_p(\Omega)}=\|f\|_{L^p(\Omega)}+\|f^\sharp_{k,\Omega}\|_{L^{p}(\Omega)}$.
\end{defn}

\begin{defn}\label{d.Tr}
    For an arbitrary open set $\Omega$ of $\mathbb{R}^n$, the trace operator $\Tr$ is defined for $u\in L^1_{loc}(\Omega)$ by
    \[
    \Tr u(x)=\lim\limits_{r\to 0}\frac{1}{\mu(\Omega\cap B_r(x))}\int_{\Omega\cap B_r(x)}u(y)dy.
    \]
    The trace operator $\Tr$ is considered for all $x\in \overline{\Omega}$ for which the limit exists.
\end{defn}

\begin{defn}[Admissible domain \cite{DRP-22}, \cite{Af-Roz}]\label{d.3}
   A domain $\Omega \subset \mathbb{R}^n$ is called admissible if it is an $n$-set, such that for $1<p<\infty$ and $k\in \mathbb{N}^*$, $W^{k}_p(\Omega)=C^k_p(\Omega)$ as sets with equivalent norms (hence, $\Omega$ is a $W^k_p$-extension domain), with a closed $d$-set boundary $\partial \Omega$, $0<d<n$, preserving local Markov's inequality.
\end{defn}

\begin{thm}[\cite{JW-84}]\label{thm-trace}
    Let $1<p<\infty$, $k\in \mathbb{N}^*$ be fixed. Let $\Omega$ be an admissible domain in $\mathbb{R}^n$. Then, for $\beta=k-(n-d)/p>0$, the following trace operators 
    \begin{enumerate}
        \item $\Tr:W^k_p(\mathbb{R}^n)\to B^{p,p}_\beta(\partial\Omega)\subset L^p(\partial\Omega)$,
        \item $\Tr_{\Omega}: W^k_p(\mathbb{R}^n)\to W^k_p(\Omega)$,
        \item $\Tr_{\partial \Omega}: W^k_p(\Omega)\to B^{p,p}_\beta (\partial \Omega)$
    \end{enumerate}
    are linear continuous and surjective with linear bounded right inverse, i.e. with bounded extension operators $E: B^{p,p}_\beta (\partial\Omega)\to W^k_p(\mathbb{R}^n)$, $E_\Omega:W^k_p(\Omega)\to W^k_p(\mathbb{R}^n)$ and $E_{\partial \Omega}: B^{p,p}_\beta(\partial \Omega)\to W^k_p(\Omega)$.
\end{thm}

For $\alpha>0$ and $1\leq p,\,q\leq \infty$ the Besov spaces in $L^p$-norm, $B^{p,q}_\alpha(F)$ are defined as follows. First we recall the space on $\mathbb{R}^n$. 
\begin{defn}[\cite{JW-84}]\label{Besov}
Let $k$ be the integer such that $0\leq k<\alpha\leq k+1$. Then the Besov space or Lipschitz space $\Lambda^{p,q}_\alpha(\mathbb{R}^n)$ consists of those $f\in L^p$ such that
\begin{equation}\label{2.1}
\sum\limits_{j\leq k}\|D^jf\|_p+\sum\limits_{j=k}\left\{\int_{\mathbb{R}^n}\frac{\|\Delta_hD^jf\|^q_p}{|h|^{n+(\alpha-k)q}}dh\right\}^{\frac{1}{q}}<\infty,
\end{equation}
if $k<\alpha<k+1$ and $1\leq q<\infty$. 
\end{defn}
If $q=\infty$, \eqref{2.1} shall be interpreted in the usual limiting way and if $\alpha=k+1$ the first difference $\Delta_h$ in \eqref{2.1} shall be replaced by the second difference $\Delta^2_h$. The norm of $f\in \Lambda^{p,q}_\alpha (\mathbb{R}^n)$ is given by \eqref{2.1} and $\Lambda^{p,q}_\alpha (\mathbb{R}^n)$ is a Banach space. If $1\leq p,\,q<\infty$, than $C^\infty_0(\mathbb{R}^n)$ is dense in $\Lambda^{p,q}_\alpha (\mathbb{R}^n)$.

We denote by $L^p(\mu;F)$ the Lebesgue space on $F$ with respect to $\mu$, with norm
\[
\|f\|_{p,\mu}=\left\{\int_F|f|^pd\mu\right\}^{1/p}.
\]
\begin{defn}[\cite{JW-84}]\label{Besov space1}
Let $F\subset \mathbb{R}^n$ be a $d$-set, $0<d\leq n$, $0<\alpha<1$ and $1\leq p\leq \infty$. A function $f$ belongs to the Besov space $B^{p,p}_\alpha(F)$ if and only if it has finite norm
    \[
    \|f\|_{B^{p,p}_\alpha(F)}=\|f\|_{p,\mu} +\left(\iint_{|x-y|<1}\frac{|f(x)-f(y)|^p}{|x-y|^{d+\alpha p}}d\mu(x)d\mu(y)\right)^{1/p}
    \]   
where $x,\,y\in F$ and $\mu$ is a $d$-measure, defined in Definition \ref{d-measure}.
\end{defn}

Note that $B^{p,p}_\alpha(\mathbb{R}^n)=\Lambda^{p,p}_\alpha(\mathbb{R}^n)$. One can also define Besov space $B^{p,q}_\alpha(F)$ where $p\neq q$. For a detailed study we refer to \cite{JW-84}.

\begin{thm}[Poincaré's inequality (\cite{DRP-22})]\label{Poincare}
Let $\Omega\subset \mathbb{R}^n$ be a bounded domain with $n\geq 2$. For all $u\in W^{1,p}_0(\Omega)$, where $1\leq p<+\infty$, we have the following inequality: 
\[
\|u\|_{L^p(\Omega)}\leq C \|\nabla u\|_{L^p(\Omega)},
\]
where the constant $C$ depends only on $p$, $q$, $n$ and $\Omega$.
Thus, the semi-norm $\|\cdot\|_{W^{1,p}_0(\Omega)}$ defined by $\|u\|_{W^{1,p}_0(\Omega)}:=\|\nabla u\|_{L^p(\Omega)}$, is a norm and is equivalent to the norm $\|\cdot\|_{W^{1,p}(\Omega)}$ on $W^{1,p}_0(\Omega)$.

Moreover, if $\Omega$ is a bounded $n$-set and $W^{1}_p(\Omega)=C^1_p(\Omega)$ for $1<p<+\infty$, then for every $u\in W^{1,p}(\Omega)$, there exists a constant $C>0$ depending only on $\Omega$, $p$, and $n$ such that 
\[
\left\|u-\frac{1}{\lambda(\Omega)}\int_\Omega ud\lambda \right\|_{L^p(\Omega)}\leq C\|\nabla u\|_{L^p(\Omega)},
\]
where $\lambda(\Omega)$ is the $n$-dimensional Lebesgue measure of $\Omega$.
\end{thm}

\section{Homogeneous Dirichlet boundary condition}

We start by considering the wave equation for Laplace operator with spatial potential, initial data and homogeneous Dirichlet boundary condition
\begin{equation}\label{nonh-p}
    \left\{\begin{array}{l}
         \partial^2_tu-\Delta u+V(x)u=f(t,x),\quad (t,x)\in [0,T]\times\Omega, \, \Omega\subset \mathbb{R},\\
         u(0,x)=u_0(x),\quad x\in \Omega,\\
         \partial_t u(0,x)=u_1(x),\quad x\in \Omega,\\
         u(t,x)=0, \quad (t,x)\in [0,T]\times\partial \Omega,
    \end{array}\right.
\end{equation}
where $V\geq 0$, $V\in L^\infty(\Omega)$ and $f\in L^2([0,T];\Omega)$ are given functions. 

To build a weak solution of the problem \eqref{nonh-p} we use Galerkin's method, as in \cite{Evans}.

We are looking for weak solution of the problem \eqref{nonh-p} in the following sense: 
\begin{defn}\label{weak-s}
    For $f\in L^2([0,T];L^2(\Omega))$, $u_0\in H^1_0(\Omega)$ and $u_1\in L^2(\Omega)$, we define $u\in L^2([0,T];H^1_0(\Omega))$ with $\partial_tu\in L^\infty([0,T];L^2(\Omega))$ and  $\partial^2_t u\in L^2([0,T];H^{-1}(\Omega))$ as a weak solution to the problem \eqref{nonh-p} if it satisfies the initial conditions $u(0)=u_0$, $\partial_t u(0)=u_1$, and for all $v \in L^2([0,T];H^1_0(\Omega))$, the following holds:
    \begin{align}\label{weaksol}
        \int\limits_0^T\left(\langle \partial^2_t u,v \rangle_{(H^{-1}(\Omega), H^1_0(\Omega))}+(\nabla u,\nabla v)_{L^2(\Omega)}+(\sqrt{V}u,\sqrt{V}v)_{L^2(\Omega)}\right)ds \nonumber \\
        =\int\limits_0^T (f,v)_{L^2(\Omega)}ds.
    \end{align}
\end{defn}

\subsection{Galerkin approximations}
We employ Galerkin's method by selecting smooth functions $w_k=w_k(x),$ $k=1,2,\ldots$ such that 
\begin{equation}\label{ort-basis}
    \left\{w_k\right\}_{k=1}^\infty\, \text{is an orthogonal basis of }H^1_0(\Omega)
\end{equation}
and
\begin{equation}\label{ortn-basis}
    \left\{w_k\right\}_{k=1}^\infty\, \text{is an orthonormal basis of }L^2(\Omega).
\end{equation}

On the other hand, we can take $\{w_k\}_{k=1}^\infty$ as the eigenfunctions of the Laplacian $-\Delta$ on the bounded domain $\Omega\subset \mathbb{R}^n$ with homogeneous Dirichlet boundary conditions, which solve the equation
\[
-\Delta w_k=\lambda_kw_k,
\]
in the weak sense
\begin{equation}\label{3.4'}
    \forall w\in H^1_0(\Omega)\quad (\nabla w_k,\nabla w)_{L^2(\Omega)}=\lambda_k (w_k,w)_{L^2(\Omega)}.
\end{equation}

For a fixed positive integer $m$, we write
\begin{equation}\label{u_m}
     u_m(t,x):=\sum\limits_{k=1}^m d^k_m(t)w_k(x),
\end{equation}
where we determine the coefficients $d^k_m(t)$ ($0\leq t\leq T$, $k=1,\ldots,m$) as follows:
\begin{equation}\label{coef-d_m}
    d^k_m(0)=(u_0,w_k)_{L^2(\Omega)}, \quad k=1,\ldots,m,
\end{equation}
\begin{equation}\label{coef-d'_m}
    \partial_td^k_m(0)=(u_1,w_k)_{L^2(\Omega)}, \quad k=1,\ldots,m,
\end{equation}
and
\begin{equation}\label{u_m_tt}
    ( \partial^2_tu_m,w_k)_{L^2(\Omega)}+(\nabla  u_m,\nabla w_k)_{L^2(\Omega)}+(V u_m,w_k)_{L^2(\Omega)}=( f,w_k)_{L^2(\Omega)},
\end{equation}
for $0\leq t\leq T, \, k=1,\ldots,m.$
\begin{prop}\label{prop1}
    For each integer $m=1,2,\ldots,$ there exists a unique function $ u_m$ of the form \eqref{u_m} satisfying \eqref{coef-d_m}-\eqref{u_m_tt}.
\end{prop}
\begin{proof}
    Let $ u_m$ be given by \eqref{u_m}. Then taking into account \eqref{ortn-basis}, we have
    \begin{equation*}
        \left( \partial^2_tu_m, w_k\right)_{L^2(\Omega)} = {d_m^k}''(t).
    \end{equation*}
    Furthermore, we have
    \[
    e^{lk}:=\left(\nabla w_l,\nabla w_k\right)_{L^2(\Omega)},\quad g^{lk}:=\left(Vw_l,w_k\right)_{L^2(\Omega)},\quad f^{k}(t):=( f,w_k)_{L^2(\Omega)}\in L^2[0,T]
    \]
    for $l,k=1,\ldots,m.$ Consequently, relation \eqref{u_m_tt} becomes the linear system of ODE
    \begin{equation}\label{d''}
    {d^k_m}''(t)+\sum\limits_{l=1}^m\left(e^{lk}+ g^{lk}\right)d^l_m(t) = f^k(t),\quad t\in [0,T],\, k=1,\ldots m,
    \end{equation}
    with the initial conditions \eqref{coef-d_m} and \eqref{coef-d'_m}. According to the standard theory for ordinary differential equations (see, e.g., \cite{Tesch}), there exists a unique function 
    \[
    \mathbf{d}_m(t)=(d^1_m(t),\ldots,d^m_m(t)),  
    \]
    satisfying \eqref{coef-d_m}, \eqref{coef-d'_m} and solving \eqref{d''} for $0\leq t\leq T$.
\end{proof}

\subsection{Energy estimates} Now, we need some estimates, uniform in $m$ at $m\to \infty$.
\begin{thm}\label{energy estimates}
    There exists a constant $C>0$, depending only on $T$, $\Omega$, such that
    \begin{align}\label{energy est}
        \sup\limits_{0\leq t\leq T}& \left(\| \partial_tu_m\|_{L^2(\Omega)}+\| u_m\|_{H^1_0(\Omega)}+\|\sqrt{V} u_m\|_{L^2(\Omega)}\right)+\| \partial_t^2u_m\|_{L^2([0,T],H^{-1}(\Omega))}\nonumber\\
    &\leq C\left( \| f\|_{L^2([0,T];L^2(\Omega))}+\left(\|V\|^{\frac{1}{2}}_{L^\infty(\Omega)}+\|V\|_{L^\infty(\Omega)}\right)\|u_0\|_{L^2(\Omega)}\nonumber\right.\\
    &+\|u_0\|_{H^1_0(\Omega)}+\|u_1\|_{L^2(\Omega)}\Big),
    \end{align}
    with the constant $C$ independent of $m$.
\end{thm}
\begin{proof}
    Multiplying the equality \eqref{u_m_tt} by ${d^k_m}'(t)$, summing over $k=1,\ldots,m$, and taking into account \eqref{u_m}, we obtain 
    \begin{equation}\label{e1}
       (\partial_t^2u_m, \partial_tu_m)+(\nabla u_m,\nabla \partial_tu_m)+(V u_m, \partial_tu_m)=( f, \partial_tu_m)
    \end{equation}
    for a.e. $0\leq t\leq T$. For the first and second terms of \eqref{e1}, we have respectively
    $\Re(\partial_t^2u_m, \partial_tu_m)=\frac{d}{dt}\left(\frac{1}{2}\| \partial_tu_m\|^2_{L^2(\Omega)}\right)$ and $\Re(\nabla  u_m,\nabla \partial_tu_m)=\frac{d}{dt}\left(\frac{1}{2}\|\nabla  u_m\|^2_{L^2(\Omega)}\right)$. For the last term of the left hand side of \eqref{e1}, taking into account that $V\geq0$ on $\Omega$, we get
    \[
    \Re(V u_m, \partial_tu_m)=\frac{d}{dt}\left(\frac{1}{2}\left(\sqrt{V} u_m,\sqrt{V} u_m\right)\right)=\frac{d}{dt}\left(\frac{1}{2}\left\|\sqrt{V} u_m\right\|^2_{L^2(\Omega)}\right).
    \]
    Using the Cauchy-Schwartz and Young inequalities successively for the right hand side of \eqref{e1}, we find
    \[
    \Re( f, \partial_tu_m)\leq \| f\|_{L^2(\Omega)}\| \partial_tu_m\|_{L^2(\Omega)}\leq \frac{1}{2}\left(\| f\|^2_{L^2(\Omega)}+\| \partial_tu_m\|^2_{L^2(\Omega)}\right).
    \]
    Considering the last expressions, we arrive at the inequality
    \begin{align}\label{3.12}
        \frac{d}{dt}&\left(\| \partial_tu_m\|^2_{L^2(\Omega)}+\| u_m\|^2_{H^1_0(\Omega)}+\left\|\sqrt{V} u_m\right\|^2_{L^2(\Omega)}\right)\leq \| f\|^2_{L^2(\Omega)}+\| \partial_tu_m\|^2_{L^2(\Omega)}\nonumber\\
        &\leq \| f\|^2_{L^2(\Omega)}+\| \partial_tu_m\|^2_{L^2(\Omega)}+\| u_m\|^2_{H^1_0(\Omega)}+\left\|\sqrt{V} u_m\right\|^2_{L^2(\Omega)}.
    \end{align}
    Now we write 
    \begin{equation}\label{eta}
        \eta(t):=\| \partial_tu_m(t)\|^2_{L^2(\Omega)} +\| u_m(t)\|^2_{H^1_0(\Omega)}+\left\| \sqrt{V} u_m(t)\right\|^2_{L^2(\Omega)}
    \end{equation}
    and 
    \begin{equation}\label{xi}
        \xi(t):=\| f(t)\|^2_{L^2(\Omega)}.
    \end{equation}
Then from the inequality \eqref{3.12} we have
\[
\eta'(t)\leq \eta(t)+\xi(t)
\]
for $0\leq t\leq T$. Thus Gronwall's inequality leads us to the estimate
\begin{equation}\label{Gron}
    \eta(t)\leq e^{t}\left(\eta(0)+\int\limits_0^t\xi(s)ds\right), \quad 0\leq t\leq T.
\end{equation}
According to \eqref{coef-d_m} and \eqref{coef-d'_m} and  $V\geq 0$, $V\in L^\infty(\Omega)$, we obtain
\begin{align*}
\eta(0)&=\| \partial_tu_m(0)\|^2_{L^2(\Omega)} +\| u_m(0)\|^2_{H^1_0(\Omega)}+\left\| \sqrt{V} u_m(0)\right\|^2_{L^2(\Omega)}\\
&= \|u_1\|^2_{L^2(\Omega)}+\|u_0\|^2_{H^1_0(\Omega)}+\|\sqrt{V}u_0\|^2_{L^2(\Omega)}\\
&\leq \|u_1\|^2_{L^2(\Omega)}+\|u_0\|^2_{H^1_0(\Omega)}+C\|V\|_{L^\infty(\Omega)}\|u_0\|^2_{L^2(\Omega)}.
\end{align*}
Thus, formulas \eqref{eta}-\eqref{Gron} provide us with
    \begin{align*}
    \| \partial_tu_m(t,\cdot)&\|^2_{L^2(\Omega)}+\|  u_m(t,\cdot)\|^2_{H^1_0(\Omega)}+\left\|\sqrt{V(\cdot)} u_m(t,\cdot)\right\|^2_{L^2(\Omega)}\\
    &\leq e^t\left(\int_0^t\|f(s,\cdot)ds\|^2_{L^2(\Omega)}+\|u_0\|^2_{H^1_0(\Omega)}+\|u_1\|^2_{L^2(\Omega)}+C\|V\|_{L^\infty(\Omega)}\|u_0\|^2_{L^2(\Omega)}\right).
    \end{align*}
    Since $0\leq t\leq T$ is arbitrary and $T<\infty$, $e^t\leq e^T=C$, we see from this the following estimate
    \begin{align}\label{max}
    \sup\limits_{0\leq t\leq T}& \left(\| \partial_tu_m\|^2_{L^2(\Omega)}+\|\nabla  u_m\|^2_{L^2(\Omega)}+\|\sqrt{V} u_m\|^2_{L^2(\Omega)}\right)\nonumber\\
    &\lesssim \| f\|^2_{L^2([0,T];L^2(\Omega))}+\|V\|_{L^\infty(\Omega)}\|u_0\|^2_{L^2(\Omega)}+\|u_0\|^2_{H^1_0(\Omega)}+\|u_1\|^2_{L^2(\Omega)},
    \end{align}
    where the constant depends only on $\Omega$ and $T$.

    Fixing any $b\in H^1_0(\Omega)$, $\|b\|_{H^1_0(\Omega)}\leq 1$, we write $b=b_1+b_2$, where $b_1\in \spn\{w_k\}_{k=1}^m$ and $(b_2,w_k)=0$ for $k=1,\ldots m$. Note that $\|b_1\|_{H^1_0(\Omega)}\leq 1$. Then by using \eqref{u_m} and \eqref{u_m_tt}, we get
    \begin{align*}
        \langle  \partial_t^2u_m,b\rangle_{(H^{-1}(\Omega),H^1_0(\Omega))}&=( \partial_t^2u_m,b)_{L^2(\Omega)}=( \partial_t^2u_m,b_1)_{L^2(\Omega)}\\
        &=( f,b_1)_{L^2(\Omega)}-(\nabla u_m,\nabla b_1)_{L^2(\Omega)}-(V u,b_1)_{L^2(\Omega)}.
    \end{align*}
    By using the Cauchy-Schwartz inequality and the Poincaré inequality (Theorem \ref{Poincare}), we obtain
    \begin{align*}
        |\langle  \partial_t^2u_m,b\rangle_{(H^{-1}(\Omega),H^1_0(\Omega))}|\lesssim \| f\|_{L^2(\Omega)}\|b_1\|_{L^2(\Omega)}+\int_\Omega\left|\nabla u_m\cdot\overline{\nabla b_1}\right|dx+\int_\Omega\left|V u_m\overline{b_1}\right|dx\\
        \lesssim \| f\|_{L^2(\Omega)}\|b_1\|_{H^1_0(\Omega)} +\| u_m\|_{H^1_0(\Omega)}\|b_1\|_{H^1_0(\Omega)} +\|V u_m\|_{L^2(\Omega)}\|b_1\|_{H^1_0(\Omega)}
    \end{align*}
    and since $\|b_1\|_{H^1_0(\Omega)}\leq 1$, we get
    \[
    |\langle  \partial_t^2u_m,b\rangle_{(H^{-1}(\Omega),H^1_0(\Omega))}|\lesssim \| f\|_{L^2(\Omega)} +\| u_m\|_{H^1_0(\Omega)}+\|V\|_{L^\infty(\Omega)}\| u_m\|_{L^2(\Omega)}.
    \]
    By using estimate \eqref{max} and Poincaré inequality (Theorem \ref{Poincare}), we have
    \[
    \|u_m\|_{L^2(\Omega)}\leq C\|u_m\|_{H^1_0(\Omega)},
    \]
    and considering this, we deduce
    \begin{align*}
     \int\limits_0^T \| \partial_t^2u_m\|^2_{H^{-1}(\Omega)}dt &\lesssim \int\limits_0^T\left(\| f\|^2_{L^2(\Omega)} +\| u_m\|^2_{H^1_0(\Omega)}+\|V\|^2_{L^\infty(\Omega)}\| u_m\|^2_{L^2(\Omega)}\right)dt\\
     &\lesssim \| f\|^2_{L^2([0,T];L^2(\Omega))} +\|u_0\|^2_{H^1_0(\Omega)}+\|V\|^2_{L^\infty(\Omega)}\|u_0\|^2_{L^2(\Omega)}+\|u_1\|^2_{L^2(\Omega)}.   
    \end{align*}
    Combining with estimate \eqref{max}, we obtain \eqref{energy est}.
\end{proof}

We establish the following well-posedness result:
\begin{thm}\label{thm3.3}
For any bounded domain $\Omega$, there exists a unique solution to the wave equation \eqref{nonh-p} with the homogeneous Dirichlet boundary condition, as defined in Definition \ref{weak-s}.
\end{thm}
\begin{proof}
    We start to prove existence. Based on the energy estimates \eqref{energy est}, the nets $\{ \partial_tu_m\}_{m=1}^\infty$, $\{\sqrt{V} u_m\}_{m=1}^\infty$ are bounded in $L^2([0,T];L^2(\Omega))$, $\{ u_m\}_{m=1}^\infty$ is bounded in $L^2([0,T];H^1_0(\Omega))$, and $\{ \partial_t^2u_m\}_{m=1}^\infty$ is bounded in $L^2([0,T];H^{-1}(\Omega))$. Then, there exists a subsequence $\{ u_{m_l}\}_{l=1}^\infty\subset \{ u_m\}_{m=1}^\infty$ and a function $ u\in L^2([0,T];H^1_0(\Omega))$, with $\sqrt{V} u,\, \partial_tu\in L^2([0,T];L^2(\Omega))$, $ \partial^2_tu\in L^2([0,T];H^{-1}(\Omega))$, such that 
    \begin{equation}\label{existweak}
        \left\{\begin{array}{l}
              u_{m_l}\rightharpoonup  u \quad\text{weakly in } L^2([0,T];H^1_0(\Omega)) \\
              \partial_tu_{m_l}\rightharpoonup  \partial_tu \quad \text{weakly in } L^2([0,T];L^2(\Omega))\\
             \sqrt{V} u_{m_l}\rightharpoonup \sqrt{V} u \quad \text{weakly in } L^2([0,T];L^2(\Omega))\\
              \partial^2_tu_{m_l}\rightharpoonup  \partial^2_tu \quad \text{weakly in } L^2([0,T];H^{-1}(\Omega)).
        \end{array}\right.
    \end{equation}

    In the same way as in \cite{Evans}, p. 405, we show that $ u$ is a weak solution of the initial-boundary problem \eqref{nonh-p}.

    Let us prove uniqueness of the solution. By the linearity of the problem, it suffices to show that the only weak solution of \eqref{nonh-p} with $f\equiv 0$ and $u_0\equiv u_1\equiv 0$ is $u\equiv 0$. To show this, let us denote for a fixed $0\leq s\leq T$,
    \begin{equation}\label{v(s)}
    v(t):=\left\{\begin{array}{ll}
         \int_t^s u(\tau)d\tau,&0\leq t\leq s,  \\
         0,& s\leq t\leq T.
    \end{array}\right.
    \end{equation}
    Following that $v(t)\in H^1_0(\Omega)$ for all $0\leq t\leq T$, and consequently
    \[
    \int\limits_0^s\left(\langle \partial^2_tu,v\rangle_{(H^{-1}(\Omega),H^1_0(\Omega))} +(\nabla u,\nabla v)_{L^2(\Omega)}+(Vu,v)_{L^2(\Omega)}\right)dt=0.
    \]
    Since $\partial_tu(0)=v(s)=0$, integrating by parts the first term above, we get
    \[
    \int\limits_0^s\left(-(\partial_tu,\partial_tv)_{L^2(\Omega)}+(\nabla u,\nabla v)_{L^2(\Omega)}+(Vu,v)_{L^2(\Omega)}\right)dt=0.
    \]
    From \eqref{v(s)} we obtain $\partial_tv=-u$ $(0\leq t\leq s)$, so that
    \[
    \int\limits_0^s\left((\partial_tu,u)_{L^2(\Omega)}-(\nabla \partial_t v,\nabla v)_{L^2(\Omega)}-\left(\sqrt{V}\partial_t v,\sqrt{V} v\right)_{L^2(\Omega)}\right)dt=0.
    \]
    Furthermore, by using the property of the partial derivative, we obtain
    \[
    \int\limits_0^s\frac{\partial}{\partial t}\left(\frac{1}{2}(u,u)_{L_2(\Omega)}-\frac{1}{2}(\nabla v,\nabla v)_{L^2(\Omega)}-\frac{1}{2}\left(\sqrt{V}v,\sqrt{V}v\right)_{L^2(\Omega)}\right)=0,
    \]
    consequently
    \begin{equation}\label{v(s)_1}
    \|u(s)\|^2_{L^2(\Omega)}+\|\nabla v(0)\|^2_{L^2(\Omega)}+\left\|\sqrt{V}v(0)\right\|^2_{L^2(\Omega)}=0,
    \end{equation}
    since $u(0)=0$, $v(s)=0$.

    Now let us denote 
    \[
    w(t)=\int\limits_0^tu(\tau)d\tau \quad (0\leq t\leq T),
    \]
    which allows us to rewrite \eqref{v(s)_1} in the following form:
     \[
     \|u(s)\|^2_{L^2(\Omega)}+\|\nabla w(s)\|^2_{L^2(\Omega)}+\left\|\sqrt{V}w(s)\right\|^2_{L^2(\Omega)}=0
     \]
     for each $0\leq s\leq T$. This implies $u\equiv0$.
     \end{proof}
\begin{cor}\label{cor3.1} For any bounded domain $\Omega$ and for the weak solution to the equation \eqref{nonh-p}, we have the following estimates:  
\begin{align}\label{partial_tu_m}
\|\partial_tu\|_{L^{\infty}([0,T],L^2(\Omega))}&\lesssim \| f\|_{L^2([0,T];L^2(\Omega))}+\|V\|^{\frac{1}{2}}_{L^\infty(\Omega)} \|u_0\|_{L^2(\Omega)}\nonumber\\
+&\|u_0\|_{H^1_0(\Omega)}+\|u_1\|_{L^2(\Omega)},
\end{align}
    
\begin{align}\label{u_m-L^inf}
\| u\|_{L^\infty([0,T];H^1_0(\Omega))}&\lesssim \| f\|_{L^2([0,T];L^2(\Omega))}+\|V\|^{\frac{1}{2}}_{L^\infty(\Omega)} \|u_0\|_{L^2(\Omega)}\nonumber\\
&+\|u_0\|_{H^1_0(\Omega)}+\|u_1\|_{L^2(\Omega)},
\end{align}

\begin{align}\label{sqrtV-v}
\|\sqrt{V} u\|_{L^\infty([0,T];L^2(\Omega))}&\lesssim \| f\|_{L^2([0,T];L^2(\Omega))}+\|V\|^{\frac{1}{2}}_{L^\infty(\Omega)} \|u_0\|_{L^2(\Omega)}\nonumber\\
&+\|u_0\|_{H^1_0(\Omega)}+\|u_1\|_{L^2(\Omega)},
\end{align}

\begin{align}\label{partial_t^2u_m}
    \|\partial_t^2u\|_{L^2([0,T];H^{-1}(\Omega))}&\lesssim \| f\|_{L^2([0,T];L^2(\Omega))}+\|V\|_{L^\infty(\Omega)} \|u_0\|_{L^2(\Omega)}\nonumber\\
&+\|u_0\|_{H^1_0(\Omega)}+\|u_1\|_{L^2(\Omega)},
\end{align}     

\begin{align}\label{Delta-u}
\|\Delta u\|_{L^2([0,T];H^{-1}(\Omega))}&\lesssim \left(1+\|V\|^\frac{1}{2}_{L^\infty(\Omega)}\right)\left(\|f\|_{L^2([0,T];L^2(\Omega))}+\|u_0\|_{H^1_0(\Omega)}+\|u_1\|_{L^2(\Omega)}\right)\nonumber\\
&+\|V\|_{L^\infty(\Omega)}\|u_0\|_{L^2(\Omega)},  
\end{align}
where the constants depend only on $\Omega$ and $T$.
\end{cor}
\begin{proof}
Passing to limits in \eqref{energy est} as $m=m_l\to \infty$, we obtain the estimates \eqref{partial_tu_m}-\eqref{partial_t^2u_m}.

Taking the $H^{-1}(\Omega)$-norm, then integrating over the time $[0,T]$, we get from the equation that
\[
\|\Delta u\|_{L^2([0,T];H^{-1}(\Omega))}\lesssim \|\partial^2_t u\|_{L^2([0,T];H^{-1}(\Omega))} +\|Vu\|_{L^2([0,T];H^{-1}(\Omega))} +\|f\|_{L^2([0,T];H^{-1}(\Omega))}.
\]
Using $0\leq V\in L^\infty(\Omega)$ and taking into account $L^2(\Omega)\subset H^{-1}(\Omega)$, we get
\begin{align*}
    \|Vu\|_{L^2([0,T];H^{-1}(\Omega))}&\leq C\|Vu\|_{L^2([0,T];L^2(\Omega))}\leq C\|\sqrt{V}\|_{L^\infty(\Omega)}\|\sqrt{V}u\|_{L^2([0,T];L^2(\Omega))}\\
    &\leq C\sqrt{T}\|\sqrt{V}\|_{L^\infty(\Omega)}\|\sqrt{V}u\|_{L^\infty([0,T];L^2(\Omega))},
\end{align*}
since
\begin{align*}
    \|\sqrt{V}u\|^2_{L^2([0,T];L^2(\Omega))} &=\int\limits_0^T\|\sqrt{V}u(t)\|^2_{L^2 (\Omega)}dt\leq \int\limits_0^T \esssup\limits_{t\in[0,T]}\| \sqrt{V}u(t)\|^2_{L^2(\Omega)}dt\\
    &=T\cdot \|\sqrt{V}u\|^2_{L^\infty ([0,T];L^2(\Omega))},
\end{align*}
where the constants depend on $\Omega$ and $T$.
We also have
\[
\|f\|_{L^2([0,T];H^{-1}(\Omega))}\leq C \|f\|_{L^2([0,T];L^2(\Omega))},
\]
where the constant $C$ depends only on $\Omega$.
Then, using the estimates \eqref{sqrtV-v} and \eqref{partial_t^2u_m}, we obtain \eqref{Delta-u}. This proves Corollary \ref{cor3.1}.
\end{proof}

\begin{thm}\label{thm3.5}
    Assume that $V\geq 0$, $V\in L^\infty(\Omega)$, $u_0\in H^2_0(\Omega)$, $u_1\in H^1_0(\Omega)$ and $f\in H^1([0,T];L^2(\Omega))$. For any bounded domain $\Omega$, the equation \eqref{nonh-p} has a unique solution satisfying the estimates
    \begin{align}\label{3.23}
        \|\partial^2_t u\|_{L^2([0,T];L^2(\Omega))}&\lesssim \| f\|_{H^1([0,T];L^2(\Omega))}+\left(1+\|V\|^{\frac{1}{2}}_{L^\infty(\Omega)}\right)\|u_1\|_{H^1_0(\Omega)}\nonumber\\
+&\left(1+\|V\|_{L^\infty(\Omega)}\right)\|u_0\|_{H^2_0(\Omega)},
    \end{align}
\begin{align}\label{3.24+}
    \|\Delta u\|_{L^2([0,T];L^2(\Omega))}&\lesssim\left(1+\|V\|^\frac{1}{2}_{L^\infty(\Omega)}\right)\left(\|f\|_{H^1([0,T];L^2(\Omega))}+\|u_0\|_{H^1_0(\Omega)}+\|u_1\|_{H^1_0(\Omega)}\right)\nonumber\\
    +&\left(1+\|V\|_{L^\infty(\Omega)}\right)\|u_0\|_{H^2_0(\Omega)}.
\end{align}    
\end{thm}
\begin{proof}
Taking the derivative with respect to $t$ over $[0,T]$ in equation \eqref{nonh-p}, we get
\begin{equation}\label{eq_u_t}
\partial^2_t (u_t)(t,x)-\Delta(u_t)(t,x)+V(x)(u_t)(t,x)=f_t(t,x).    
\end{equation}
For the initial data, we have 
\[
u_t(0,x)=u_1(x),\quad
\partial_t(u_t)(0,x)=\partial^2_t u(0,x).
\]
Using the equation \eqref{nonh-p}, we obtain
\[
\partial^2_tu(0,x)=\Delta u_0(x)-V(x)+f(0,x).
\]
We denote $u_t(t,x)=v(t,x)$, and rewrite the equation \eqref{eq_u_t} in the following form:
\begin{equation}\label{eq_v}
    \left\{\begin{array}{l}
         \partial^2_tv(t,x)-\Delta v(t,x)+V(x)v(t,x)=f_t(t,x), \quad (t,x)\in [0,T]\times \Omega,\\
    v(0,x)=u_1(x), \quad x\in \Omega,\\
    \partial_t v(0,x)=\Delta u_0(x)-V(x)u_0(x)+f(0,x), \quad x\in \Omega,  \\
        v(t,x)=0,\quad (t,x)\in [0,T]\times \partial\Omega.
    \end{array}\right.    
\end{equation}
Then, assuming $u_0\in H^2_0(\Omega)$, $f\in H^1([0,T];L^2(\Omega))$ and using the estimate \eqref{partial_tu_m}, we obtain
\begin{align*}
\|\partial_tv\|_{L^{\infty}([0,T],L^2(\Omega))}&\lesssim \| f\|_{H^1([0,T];L^2(\Omega))}+\|V\|^{\frac{1}{2}}_{L^\infty(\Omega)} \|u_1\|_{L^2(\Omega)}\nonumber\\
+&\|u_1\|_{H^1_0(\Omega)}+\|u_0\|_{H^2_0(\Omega)}+\|V\|_{L^\infty(\Omega)}\|u_0\|_{L^2(\Omega)}+\|f(0,\cdot)\|_{L^2(\Omega)},
\end{align*}
since the norm $\| f\|_{H^1([0,T];L^2(\Omega))}$ dominates $\|f(0,\cdot)\|_{L^2(\Omega)}$ by the Sobolev embedding it, we obtain
\begin{align*}
\|\partial_tv\|_{L^{\infty}([0,T],L^2(\Omega))}&\lesssim \| f\|_{H^1([0,T];L^2(\Omega))}+\|V\|^{\frac{1}{2}}_{L^\infty(\Omega)} \|u_1\|_{L^2(\Omega)}\nonumber\\
+&\|u_1\|_{H^1_0(\Omega)}+\|u_0\|_{H^2_0(\Omega)}+\|V\|_{L^\infty(\Omega)}\|u_0\|_{L^2(\Omega)}.
\end{align*}
Since $v=u_t$ and $L^\infty([0,T];L^2(\Omega))\subset L^2([0,T];L^2(\Omega))$, we get 
\begin{align*}
\|\partial^2_tu\|_{L^2([0,T];L^2(\Omega))}&\lesssim \|\partial^2_tu\|_{L^{\infty}([0,T],L^2(\Omega))}=\|\partial_tv\|_{L^{\infty}([0,T],L^2(\Omega))},
\end{align*}
which implies \eqref{3.23}.

Similarly as in Corollary \ref{cor3.1}, by using equation \eqref{nonh-p} and taking $L^2$-norm and then integration over the time $[0,T]$, we obtain
\[
\|\Delta u\|_{L^2([0,T];L^2(\Omega))}\lesssim \|\partial^2_t u\|_{L^2([0,T];L^2(\Omega))} +\|Vu\|_{L^2([0,T];L^2(\Omega))} +\|f\|_{L^2([0,T];L^2(\Omega))}.
\]
By using the estimates \eqref{sqrtV-v} and \eqref{3.23}, we get \eqref{3.24+}. This proves Theorem \ref{thm3.5}.
\end{proof}

\begin{defn}[Weak Laplacian \cite{DRP-22}]\label{weak Laplace}
    The Laplacian operator $-\Delta$ on $\Omega$ is considered with homogeneous Dirichlet boundary conditions in the weak sense:
    \begin{align*}
    -\Delta : \mathcal{D} (-\Delta) \subset H^1_0(\Omega) \to L^2(\Omega)\\
    u\mapsto -\Delta u.
    \end{align*}
Here $\mathcal{D}(-\Delta)$ is the domain of $-\Delta$, defined in the following way: $u\in \mathcal{D}(-\Delta)$ if and only if $u\in H^1_0(\Omega)$ and $-\Delta u\in L^2(\Omega)$ in the sense that there exists $f\in L^2(\Omega)$ such that 
\[
\forall v \in H^1_0(\Omega) \,\, (\nabla u,\nabla v)_{L^2(\Omega)}=(f,v)_{L^2(\Omega)}.
\]
The operator $-\Delta$ is linear self-adjoint and coercive in the sense that for $u\in \mathcal{D}(-\Delta)$
\[
(-\Delta u,u)_{L^2(\Omega)}=(\nabla u,\nabla u)_{L^2(\Omega)}
\]
and we use the notation 
\[
\|u\|_{\mathcal{D}(-\Delta)}=\|\Delta u \|_{L^2(\Omega)}.
\]

\end{defn}
We need to include the following space as well:
\begin{defn}\label{def_9}
    For $U$ an open set we define the Hilbert space 
    \[
    M(U)=L^2([0,T];\mathcal{D}(-\Delta))\cap H^2([0,T];L^2(U))
    \]
    with the norm
    \[
    \|u\|^2_{M(U)}=\|\Delta u\|^2_{L^2([0,T];L^2(U))}+\|\partial^2_tu\|^2_{L^2([0,T];L^2(U))}+\|\sqrt{V}u\|^2_{L^2([0,T];L^2(U))}
    \]
    associated to the scalar product
    \[
    (u,v)_{M(U)}=\int\limits_0^T\left((\Delta u,\Delta v)_{L^2(U)}+(\partial^2_t u,\partial^2_t v)_{L^2(U)}+(\sqrt{V} u,\sqrt{V} v)_{L^2(U)}\right)ds.
    \]
\end{defn}

\begin{thm}\label{thm3.5+}
    For an arbitrary bounded domain $\Omega\subset \mathbb{R}^n$ and
    \begin{equation}\label{3.23+}
        M_0(\Omega):=\{u\in M(\Omega)|u(0)=0,\,\partial_tu(0)=0\}
    \end{equation}
    there exists a unique weak solution to the  homogeneous Dirichlet problem for the wave equation with potential and homogeneous initial conditions, in terms of the variational formulation \eqref{weaksol}, and $u\in M_0(\Omega)$ if and only if $f\in H^1([0,T];L^2(\Omega))$.

    Moreover, the solution satisfies the estimate
    \begin{equation}\label{3.24'}
        \|u\|_{M(\Omega)}\lesssim \|f\|_{H^1([0,T];L^2(\Omega))}.
    \end{equation}
\end{thm}
\begin{proof}
    We follow the idea of the proof Theorem 13 in \cite{DRP-22}. If the source term $f\in H^1([0,T];L^2(\Omega))$, $V\in L^\infty(\Omega)$, $u_0\in H^2_0(\Omega)$ and $u_1\in H^1_0(\Omega)$, by Theorem \ref{energy estimates}, Corollary \ref{cor3.1} and Theorem \ref{thm3.5} there exists a unique weak solution to the problem \eqref{nonh-p} in terms of \eqref{weaksol} with $\Delta u\in L^2([0,T];L^2(\Omega))$, $\partial^2_tu\in L^\infty([0,T],H^2(\Omega))$, and $\sqrt{V}u\in L^\infty([0,T];L^2(\Omega))$. Hence $u\in M(\Omega)$. 
    
    Let us consider a weak solution $u\in L^2([0,T];\mathcal{D}(-\Delta))\cap H^2([0,T];L^2(\Omega))$ with the initial conditions $u(0)=0$, $\partial_tu(0)=0$ of the homogeneous problem \eqref{nonh-p} for the wave equation. Due to the linearity of the problem, $u$ is unique. Additionally, given the regularity of $u$, it follows that $f\in H^1([0,T];L^2(\Omega))$.
\end{proof}

\section{Non-homogeneous Dirichlet condition}
In this section we consider the wave equation with potential with initial data and non-homogeneous boundary condition
\begin{equation}\label{wave}
\left\{\begin{array}{l}
     \Tilde{u}_{tt}-\Delta \Tilde{u} +V(x)\Tilde{u}=f(t,x), \quad (t,x)\in [0,T]\times\Omega \subset \mathbb{R}^n,  \\
      \Tilde{u}(0,x)=u_0(x),\quad x\in \Omega,\\
      \Tilde{u}_t(0,x)=u_1(x), \quad x\in \Omega,\\
      \Tilde{u}(t,x)=g(t,x),\quad (t,x)\in [0,T]\times\partial \Omega.
\end{array}\right.
\end{equation} 

\begin{thm}\label{th4.1} Let $\Omega$ be an admissible domain in $\mathbb{R}^n$ with a $d$-set boundary $\partial \Omega$ for $n-2<d<n$.
Let also $\beta=2-\frac{n-d}{2}>0$ and
\begin{equation}\label{4.1'}
F:=H^2([0,T];B^{2,2}_\beta(\partial\Omega)).
\end{equation}
    For $u_0\in H^2_0(\Omega)$, $u_1\in H^1_0(\Omega)$, $g\in F$, $V\in L^\infty(\Omega)$, $V\geq 0$ and $f\in H^1([0,T];L^2(\Omega))$ with the consistency conditions
    \[
    g(0)=\Tr_{\partial\Omega}u_0,\quad \partial_tg(0)=\Tr_{\partial \Omega}u_1,
    \]
    there exists a unique weak solution $\Tilde{u}$ of the problem \eqref{wave}. It is a weak solution in such a way that $\Tilde{u}=u^*+G$ with
$$G=G(t,x)\in M_1(\Omega):=H^2([0,T];L^2(\Omega))\cap L^2([0,T];H^2(\Omega))$$ 
such that $\Tr_{\partial \Omega}G(t)=g(t)$ and with $u^*\in M(\Omega)$, which is the unique weak solution of the system
\begin{equation}\label{4.2}
    \left\{\begin{array}{l}
         \partial_t^2u-\Delta u+V u=f-\partial^2_tG+\Delta G-VG,\\
         u(0)=u_0-G(0,x),\, \partial_tu(0)=u_1-\partial_tG(0,x),\quad x\in \Omega, 
    \end{array}\right.
\end{equation}
in the sense of variational formulation \eqref{u_m_tt}. Moreover, we have an apriori estimate
\begin{align}\label{4.3}
 \|u^*\|_{M(\Omega)}&\lesssim\left(1+\|V\|^\frac{1}{2}_{L^\infty(\Omega)}\right)\left(\|f\|_{H^1([0,T];L^2(\Omega))}+\|u_0\|_{H^1_0(\Omega)}+\|u_1\|_{H^1_0(\Omega)}\right)\nonumber\\
&+\left(1+\|V\|_{L^\infty(\Omega)} \right)\|u_0\|_{H^2_0(\Omega)} +\|g\|_{F}.
\end{align}
\end{thm}
\begin{proof}
Since $g\in F$, the existence of $G\in M_1(\Omega)$ with $\Tr_{\partial\Omega}G=g$ comes from properties of the trace operator, which according to Theorem \ref{thm-trace} has a bounded linear right inverse, i.e. the extension operator $E_{\partial\Omega}$ such as
\begin{align*}
E_{\partial \Omega}: B^{2,2}_\beta(\partial\Omega)\to H^2(\Omega).
\end{align*}
In addition, the boundedness of $E_{\partial\Omega}$ implies
\begin{align*}
    \|E_{\partial\Omega}g\|_{M_1(\Omega)}&=\|\partial^2_t E_{\partial\Omega}g\|_{L^2([0,T];L^2(\Omega))}+ \|E_{\partial\Omega}g\|_{L^2([0,T];H^2(\Omega))}\\
    &\leq \|\partial^2_tg\|_{L^2([0,T];B^{2,2}_{\beta}(\partial \Omega))}+\|g\|_{L^2([0,T];B^{2,2}_\beta(\partial\Omega))}\leq \|g\|_{H^2([0,T];B^{2,2}_\beta(\partial\Omega))}
\end{align*}
since
\[
\|\partial^2_tE_{\partial\Omega}g\|_{L^2([0,T];L^2(\Omega))}\leq C\|\partial^2_tE_{\partial\Omega}g\|_{L^2([0,T];H^2(\Omega))}.
\]
Therefore
\[
\|G\|_{M_1(\Omega)}=\|E_{\partial \Omega}g\|_{M_1(\Omega)}\leq C\|g\|_{F}.
\]

Let $u^*$ be the solution of problem \eqref{4.2}. The regularity of $G$ ensures that
\[
\Delta G-\partial^2_tG-VG\in L^2([0,T];L^2(\Omega)),\quad u_0-G(0,x)\in H^2(\Omega),\quad u_1-\partial_tG(0,x)\in H^1_0(\Omega),
\] 
because the consistency conditions allow us to have
\[
\Tr_{\partial\Omega}(u_0-G(0,x))=\Tr_{\partial\Omega}(u_0)-g(0)=0,
\]
\[
\Tr_{\partial\Omega}(u_1-\partial_t G(0,x))=\Tr_{\partial\Omega}(u_1)-\partial_t g(0)=0.
\]
Then we apply Corollary \ref{cor3.1} to obtain the existence of a unique solution $u^*$ of problem \eqref{4.2} with the desired regularity. The regularity of $u_0$, $u_1$ and $G$ allows to deduce estimate \eqref{4.3}. 
\end{proof}

\section{Very weak solutions}
In this section we consider the wave equation when potential, initial data and boundary condition are less regular in the sense of so called ``$\delta$-function" distributions. For this, we will be using the concept of very weak solutions. 

Assume that the potential $V$, initial data $(u_0,u_1)$ are distributions on $\Omega$, and the boundary condition $g$ is distribution on $\partial \Omega$.

\begin{defn}\label{def5.1}
    (i) A net of functions $(u_\varepsilon=u_\varepsilon(t,x))$ is said to be $L^2$-moderate if there exist $N\in \mathbb{N}$ and $C>0$ such that 
    \[
    \|u_\varepsilon\|_{L^2([0,T];L^2(\Omega))}\leq C\varepsilon^{-N}.
    \]
    (ii) A net of functions $(V_\varepsilon=V_\varepsilon(x))$ is said to be $L^\infty$-moderate if there exist $N\in \mathbb{N}$ and $C>0$ such that 
    \[
    \|V_\varepsilon\|_{L^\infty(\Omega)}\leq C\varepsilon^{-N}.
    \]
\end{defn}

\begin{defn}\label{def14}
(i) A net of functions $(f_\varepsilon=f_\varepsilon(t,x))$ is said to be $H^1(L^2)$-moderate if there exist $N\in \mathbb{N}$ and $C>0$ such that
\[
\|f_\varepsilon\|_{H^1([0,T];L^2(\Omega))}\leq C e^{-N}.
\]
(ii) A net of functions $(u_{0,\varepsilon}=u_{0,\varepsilon}(x))$ is said to be $H^2_0$-moderate if there exist $N\in \mathbb{N}$ and $C>0$ such that
\[
\|u_{0,\varepsilon}\|_{H^2_0(\Omega)}\leq C e^{-N}.
\]
(iii) A net of functions $(u_{1,\varepsilon}=u_{1,\varepsilon}(x))$ is said to be $H^1_0$-moderate if there exist $N\in \mathbb{N}$ and $C>0$ such that
\[
\|u_{1,\varepsilon}\|_{H^1_0(\Omega)}\leq C e^{-N}.
\]
\end{defn}

\begin{rem}\label{rem5.1}
 We note that such assumptions are natural for distributional coefficients in the sense that regularizations of distributions are moderate. Precisely, by the structure theorems for distributions (see, e.g. \cite{Friedlander}), we know that distributions 
\begin{equation}\label{moder}
  \mathcal{D}'(\Omega) \subset \{L^\infty(\Omega) -\text{moderate families} \},  
\end{equation}
and we see from \eqref{moder}, that a solution to an initial/boundary problem may not exist in the sense of distributions, while it may exist in the set of $L^\infty$-moderate functions. 
\end{rem}

As an example, let us take $f\in L^2(\Omega)$, $f:\Omega\to \mathbb{C}$. We introduce the function
    \[
    \Tilde{f}=\left\{\begin{array}{l}
    f,\,\text{on }\Omega,\\
    0,\, \text{on } \mathbb{R}^n \setminus \Omega,
    \end{array}\right.
    \]
    then $\Tilde{f}:\mathbb{R}^n\to \mathbb{C}$, and $\Tilde{f}\in \mathcal{E}'(\mathbb{R}^n)$.

    Let $\Tilde{f}_\varepsilon=\Tilde{f}*\psi_\varepsilon$ be obtained as the convolution of $\Tilde{f}$ with a Friedrich mollifier $\psi_\varepsilon$, where 
$$\psi_\varepsilon(x)=\frac{1}{\varepsilon}\psi\left(\frac{x}{\varepsilon}\right),\quad \text{for}\,\, \psi\in C^\infty_0(\mathbb{R}^n),\, \int \psi=1. $$
Then the regularising net $(\Tilde{f}_\varepsilon)$ is $L^p$-moderate for any $p \in \mathbb{R}^n\setminus \Omega$, and it approximates $f$ on $\Omega$:
$$0\leftarrow \|\Tilde{f}_\varepsilon-\Tilde{f}\|^p_{L^p(\mathbb{R}^n)}\approx \|\Tilde{f}_\varepsilon-f\|^p_{L^p(\Omega)}+\|\Tilde{f}_\varepsilon\|^p_{L^p(\mathbb{R}^n\setminus \Omega)}.$$

Now, introduce the notion of a very weak solution to the wave equation for Laplacian with potential with initial data and non-homogeneous Dirichlet boundary condition
\begin{equation}\label{5.2}
    \left\{\begin{array}{l}
    \partial^2_tu(t,x)-\Delta u(t,x)+V(x)u(t,x)=f(t,x),\quad (t,x)\in[0,T]\times \Omega,\\
    u(0,x)=u_0(x),\quad x\in \Omega,\\
    u_t(0,x)=u_1(x),\quad x\in \Omega,\\
    u(t,x)=g(t,x),\quad (t,x)\in [0,T]\times \partial \Omega.
    \end{array}\right.
\end{equation}

\begin{defn}\label{defn5.2}
    Let $V\in \mathcal{D}'(\Omega)$.  The net $(u_\varepsilon)_{\varepsilon>0}$ is said to be a very weak solution to the initial/boundary problem  \eqref{5.2} if there exists an $L^\infty$-moderate regularization $V_\varepsilon$ of $V$, $H^1(L^2)$-moderate regularization $f_\varepsilon$ of $f$,  $H^2_0$-moderate regularization $u_{0,\varepsilon}$ of $u_0$ and $H^1_0$-moderate regularization $u_{1,\varepsilon}$ of $u_1$, such that
\begin{equation}\label{5.3}
\left\{\begin{array}{l}\partial^2_t u_\varepsilon(t,x)-\Delta u_\varepsilon(t,x)+V_\varepsilon(x) u_\varepsilon(t,x)=f_\varepsilon(t,x),\quad (t,x)\in [0,T]\times\Omega,\\
 u_\varepsilon(0,x)=u_{0,\varepsilon}(x),\,\,\, x\in \Omega, \\
\partial_t u_\varepsilon(0,x)=u_{1,\varepsilon}(x), \,\,\, x\in \Omega,\\
u_\varepsilon(t,x)=g(t,x), \quad (t,\Omega)\in[0,T]\times\partial\Omega,
\end{array}\right.\end{equation}
and $(u_\varepsilon)$ is $M$-moderate.
\end{defn}

According to the notions above, we have the following properties of very weak solutions.
\begin{thm}[Existence]\label{thm5.1}
    Let the coefficient $V\geq 0$ and initial data $(u_0,\, u_1)$ be distributions in $\Omega$, and the boundary condition $g\in F$. Then the initial/boundary problem  \eqref{5.2} has a very weak solution.
\end{thm}
\begin{proof}
Since the formulation of \eqref{5.2} in this case might be impossible in the distributional sense due to issues related to the product of distributions, we replace \eqref{5.2} with a regularised equation. That is, we regularize $V$ by some set $V_{\varepsilon}$ of functions from $L^\infty(\Omega)$, $u_0$ by some set of functions $u_{0,\varepsilon}$ from $H^2_0(\Omega)$, $u_1$  by some set of the functions $u_{1,\varepsilon}$ from $ H^1_0(\Omega)$ and $f$ by some set $f_\varepsilon$ of the functions from $H^1([0,T];L^2(\Omega))$.

Hence, $V_\varepsilon$, $u_{0,\varepsilon}$ and $u_{1,\varepsilon}$ are $L^\infty$, $H^2_0$ and $H^1_0$-moderate regularizations of the coefficient $V$ and the Cauchy data $(u_0,u_1)$ respectively and $f_\varepsilon$ is $H^1(L^2)$-moderate regularization of the function $f$. Therefore, by Definition \ref{def5.1} there exist $N\in \mathbb{N}_0$, $C_1>0$, $C_2>0$, $C_3>0$ and $C_4>0$ such that
$$\|V_\varepsilon\|_{L^\infty(\Omega)}\leq C_1\varepsilon^{-N},\quad \|u_{0,\varepsilon}\|_{H^2_0(\Omega)}\leq C_2\varepsilon^{-N}, \quad \|u_{1,\varepsilon}\|_{H^1_0(\Omega)}\leq C_3\varepsilon^{-N},$$
$$\|f_\varepsilon\|_{H^1([0,T];L^2(\Omega))}\leq C_4\varepsilon^{-N}.$$

We fix $\varepsilon\in (0,1]$, and consider the regularised problem \eqref{5.3}. Then all discussions and calculations of Theorem \ref{th4.1} are valid. Thus, by Theorem \ref{th4.1}, the equation \eqref{5.2} has unique weak solution $u_\varepsilon(t,x)$ in the space $M(\Omega)$ and there exist $N\in \mathbb{N}_0$ and $C>0$, such that
\begin{align*} \|u_\varepsilon\|_{M(\Omega)}&\lesssim\left(1+\|V_{\varepsilon}\|^\frac{1}{2}_{L^\infty(\Omega)}\right)\left(\|f_{\varepsilon}\|_{H^1([0,T];L^2(\Omega))}+\|u_{0,\varepsilon}\|_{H^1_0(\Omega)}+\|u_{1,\varepsilon}\|_{H^1_0(\Omega)}\right)\nonumber\\
&+\left(1+\|V_{\varepsilon}\|_{L^\infty(\Omega)} \right)\|u_{0,\varepsilon}\|_{H^2(\Omega)} +\|g\|_{F}\leq C \varepsilon^{-N},
\end{align*}
where the constant in this inequality is independent of $u_{0,\varepsilon}$, $u_{1,\varepsilon}$, $V_\varepsilon$, $f_\varepsilon$ and $g$.
Hence, $(u_\varepsilon)$ is moderate, and the proof of Theorem \ref{thm5.1} is complete.
\end{proof}

\begin{defn}[Negligibility]\label{def5.3}
Let $(u_\varepsilon)$, $(\Tilde{u}_\varepsilon)$ be two nets in $L^2(\Omega)$. Then, the net $(u_\varepsilon-\Tilde{u}_\varepsilon)$ is called $L^2$-negligible, if for every $N\in \mathbb{N}$ there exist $C>0$ such that the following condition is met
$$\|u_\varepsilon(t,\cdot)-\Tilde{u}_\varepsilon(t,\cdot)\|_{L^2(\Omega)}\leq C \varepsilon^N,$$
for all $\varepsilon\in (0,1]$ uniformly in $t\in [0,T]$. The constant $C$ can depend on $N$ but not on $\varepsilon$.
\end{defn}

Let us consider the ``$\varepsilon$-parameterised problems":
\begin{equation}\label{5.4}
    \left\{\begin{array}{l}\partial^2_t u_\varepsilon(t,x)-\Delta u_\varepsilon(t,x)+V_\varepsilon(x) u_\varepsilon(t,x)=f_\varepsilon(t,x),\quad (t,x)\in [0,T]\times\Omega,\\
 u_\varepsilon(0,x)=u_{0,\varepsilon}(x),\,\,\, x\in \Omega, \\
\partial_t u_\varepsilon(0,x)=u_{1,\varepsilon}(x), \,\,\, x\in \Omega,\\
u_\varepsilon(t,x)=g(t,x),\quad (t,x)\in[0,T]\times\partial\Omega,
\end{array}\right.
\end{equation}
and
\begin{equation}\label{5.5}
    \left\{\begin{array}{l}\partial^2_t \Tilde{u}_\varepsilon(t,x)-\Delta \Tilde{u}_\varepsilon(t,x)+\Tilde{V}_\varepsilon(x) \Tilde{u}_\varepsilon(t,x)=\Tilde{f}_\varepsilon(t,x),\quad (t,x)\in [0,T]\times\Omega,\\
 \Tilde{u}_\varepsilon(0,x)=\Tilde{u}_{0,\varepsilon}(x),\,\,\, x\in \Omega, \\
\partial_t \Tilde{u}_\varepsilon(0,x)=\Tilde{u}_{1,\varepsilon}(x), \,\,\, x\in \Omega,\\
\Tilde{u}_\varepsilon(t,x)=g(t,x), \quad (t,x)\in[0,T]\times\partial\Omega.
\end{array}\right.
\end{equation}

\begin{defn}[Uniqueness of the very weak solution]\label{def5.4}
Let $V\in \mathcal{D}'(\Omega)$. We say that initial/boundary problem \eqref{5.2} has an unique very weak solution, if
for all $L^\infty$-moderate nets $V_\varepsilon$, $\Tilde{V}_\varepsilon$, such that $(V_\varepsilon-\Tilde{V}_\varepsilon)$ is $L^\infty$-negligible; and for all $H^1_0$-moderate regularizations $u_{1,\varepsilon},\,\Tilde{u}_{1,\varepsilon}$ such that $(u_{1,\varepsilon}-\Tilde{u}_{1,\varepsilon})$ is $H^1_0$-negligible, $H^1(L^2)$-moderate regularizations $f_{\varepsilon},\,\Tilde{f}_{\varepsilon}$ such that $(f_{\varepsilon}-\Tilde{f}_{\varepsilon})$ is $H^1(L^2)$-negligible, for all $u_{0,\varepsilon},\,\Tilde{u}_{0,\varepsilon}$ such that $(u_{0,\varepsilon}-\Tilde{u}_{0,\varepsilon})$ is $H^2_0$-negligible, we have that $u_\varepsilon-\Tilde{u}_\varepsilon$ is $M$-negligible.
\end{defn}

\begin{thm}[Uniqueness of the very weak solution]\label{thm5.2}
Assume that the boundary condition $g\in F(\Omega)$ and the coefficient $V\geq 0$, initial data $(u_0,\, u_1)$, source term $f$ be distributions in $\Omega$. Then the very weak solution to the initial/boundary problem  \eqref{5.2} is unique.
\end{thm}

\begin{proof}
Let us denote by $u_\varepsilon$ and $\Tilde{u}_\varepsilon$ the families of solutions for the initial/boundary problems \eqref{5.4} and \eqref{5.5} correspondingly. Let $U_\varepsilon$ represent the difference between these nets $U_\varepsilon(t,\cdot):=u_\varepsilon(t,\cdot)-\Tilde{u}_\varepsilon(t,\cdot)$ then $U_\varepsilon$ solves the equation
\begin{equation}\label{5.6}
    \left\{\begin{array}{l}\partial^2_t U_\varepsilon(t,x)-\Delta U_\varepsilon(t,x)+V_\varepsilon(x) U_\varepsilon(t,x)=F_\varepsilon(t,x),\quad (t,x)\in [0,T]\times\Omega,\\
 U_\varepsilon(0,x)=(u_{0,\varepsilon}-\Tilde{u}_{0,\varepsilon})(x),\,\,\, x\in \Omega, \\
\partial_t U_\varepsilon(0,x)=(u_{1,\varepsilon}-\Tilde{u}_{1,\varepsilon})(x), \,\,\, x\in \Omega,\\
U_\varepsilon(t,x)=0,\,\,(t,x)\in[0,T]\times\partial\Omega,
\end{array}\right.
\end{equation}
where we set $F_\varepsilon(t,x):=(\Tilde{V}_\varepsilon(x)-V_\varepsilon(x))\Tilde{u}_\varepsilon(t,x)+f_\varepsilon(t,x)-\Tilde{f}_\varepsilon(t,x)$ for the source term to the initial/boundary problem \eqref{5.6}.

Taking the $M(\Omega)$-norm of the $U_\varepsilon$ by using \eqref{4.3}, we obtain
\begin{align*}
 \|U_\varepsilon\|_{M(\Omega)}&\lesssim\left(1+\|V_{\varepsilon}\|^\frac{1}{2}_{L^\infty(\Omega)}\right)\left(\|F_{\varepsilon}\|_{H^1([0,T];L^2(\Omega))}+\|U_{\varepsilon}(0,\cdot)\|_{H^1_0(\Omega)}\right.\\
&\left.+\|\partial_tU_{\varepsilon}(0,\cdot)\|_{H^1_0(\Omega)}\right)+\left(1+\|V_{\varepsilon}\|_{L^\infty(\Omega)} \right)\|U_{\varepsilon}(0,\cdot)\|_{H^2_0(\Omega)}.
\end{align*}

For the source term $F_\varepsilon$, we have
\[
\|F_\varepsilon\|_{H^1([0,T];L^2(\Omega))}\leq \|V_\varepsilon-\Tilde{V}_\varepsilon\|_{L^\infty}\|\Tilde{u}_\varepsilon\|_{H^1([0,T];L^2(\Omega))}+\|f_\varepsilon-\Tilde{f}_\varepsilon\|_{H^1([0,T];L^2(\Omega))},
\]
and combining the initial data of \eqref{5.6}, we obtain
\begin{align*}
 \|U_\varepsilon\|_{M(\Omega)}&\lesssim \left(1+\|V_{\varepsilon}\|^\frac{1}{2}_{L^\infty(\Omega)}\right) \left(\|V_\varepsilon-\Tilde{V}_\varepsilon\|_{L^\infty}\|\Tilde{u}_\varepsilon\|_{H^1([0,T];L^2(\Omega))}\right.\\
&+\left.\|f_\varepsilon-\Tilde{f}_\varepsilon\|_{H^1([0,T];L^2(\Omega))}+\|u_{0,\varepsilon}-\Tilde{u}_{0,\varepsilon}\|_{H^1_0(\Omega)}+\|u_{1,\varepsilon}-\Tilde{u}_{1,\varepsilon}\|_{H^1_0(\Omega)}\right)\\
&+\left(1+\|V_{\varepsilon}\|_{L^\infty(\Omega)} \right)\|u_{0,\varepsilon}-\Tilde{u}_{0,\varepsilon}\|_{H^2_0(\Omega)} .
\end{align*}
Considering the negligibility of the nets $V_\varepsilon-\Tilde{V}_\varepsilon$, $u_{0,\varepsilon}-\Tilde{u}_{0,\varepsilon}$, $u_{1,\varepsilon}-\Tilde{u}_{1,\varepsilon}$ and $f_\varepsilon-\Tilde{f}_\varepsilon$, we get
\begin{align*}
\|U_\varepsilon\|_{M(\Omega)}&\leq \left(1+C_1\varepsilon^{-N_1}\right)\left(C_2\varepsilon^{N_2}\varepsilon^{-N_3}+C_3\varepsilon^{N_4}+C_4\varepsilon^{N_5}+C_5\varepsilon^{N^6}\right)\\
&+\left(1+C_6\varepsilon^{-N_7}\right)\varepsilon^{N_8}
\end{align*}
for some $C_1>0$, $C_2>0$, $C_3>0$, $C_4>0$, $C_5>0$, $C_6>0$, $N_1,\,N_3,\,N_7\in \mathbb{N}$ and all $N_2,\,N_4,\,N_5,\,N_6,\,N_8\in \mathbb{N}$. Then, for some $C_K>0$ and all $K\in\mathbb{N}$ 
\[\|U_\varepsilon\|_{M(\Omega)}\leq C_K \varepsilon^K,\]
and this completes the proof of Theorem \ref{thm5.2}.
\end{proof}

\begin{thm}[Consistency]\label{thm5.3} Assume that $0\leq V\in L^\infty(\Omega)$, $g\in F(\Omega)$ and let $(V_\varepsilon)$ be any $L^\infty$-regularization of $V$, that is $\|V_\varepsilon-V\|_{L^\infty(\Omega)}\to 0$ as $\varepsilon\to 0$. Let the initial data satisfy $(u_0,\, u_1) \in H^2_0(\Omega)\times H^1_0(\Omega)$ and the source term $f\in H^1([0,T];L^2(\Omega))$. Let $u$ be a very weak solution of the initial/boundary problem \eqref{5.2}. Then for any families $V_\varepsilon$, $u_{0,\varepsilon}$, $u_{1,\varepsilon}$, $f_\varepsilon$ and $g_\varepsilon$ such that $\|u_{0}-u_{0,\varepsilon}\|_{H^2_0(\Omega)}\to 0$, $\|u_{1}-u_{1,\varepsilon}\|_{H^1_0(\Omega)}\to 0$, $\|V-V_{\varepsilon}\|_{L^\infty(\Omega)}\to 0$, $\|f-f_\varepsilon\|_{H^1([0,T];L^2(\Omega))}\to0$ as $\varepsilon\to 0$, any representative $(u_\varepsilon)$ of the very weak solution converges as 
$$\|u-u_\varepsilon\|_{M(\Omega)}\to 0$$ 
for $\varepsilon\to 0$ to the unique classical solution $u\in M(\Omega)$ of the initial/boundary problem \eqref{5.2} given by Theorem \ref{th4.1}.
\end{thm}

\begin{proof}
Let us denote $\Tilde{U}_\varepsilon$ as the difference between $u$ and $u_\varepsilon$, that is
\[\Tilde{U}_\varepsilon(t, x):= u(t,x)-u_\varepsilon(t,x).\]
Then the net $\Tilde{U}_\varepsilon$ is a solution to the initial/boundary problem
\begin{equation}\label{5.7}
    \left\{\begin{array}{l}
        \partial^2_t\Tilde{U}_\varepsilon(t,x)-\Delta \Tilde{U}_\varepsilon(t,x)+V_\varepsilon(x)\Tilde{U}_\varepsilon(t,x)=F_\varepsilon(t,x),\quad (t,x)\in [0,T]\times \Omega,\\
        \Tilde{U}_\varepsilon(0,x)=(u_0-u_{0,\varepsilon})(x),\quad x\in \Omega,\\
        \partial_t\Tilde{U}_\varepsilon(0,x)=(u_1-u_{1,\varepsilon})(x),\quad x\in \Omega,\\
        \Tilde{U}_\varepsilon(t,x)=0, \quad (t,x)\in[0,T]\times\partial\Omega,
    \end{array}\right.
\end{equation}
where $F_\varepsilon(t,x)=(V_\varepsilon(x)-V(x))u(t,x)+f(t,x)-f_\varepsilon(t,x)$. Analogously to Theorem \ref{5.2} we have that
\begin{align*}
 \|\Tilde{U}_\varepsilon\|_{M(\Omega)}&\lesssim \left(1+\|V\|^\frac{1}{2}_{L^\infty(\Omega)}\right) \left(\|V-V_\varepsilon\|_{L^\infty}\|u\|_{H^1([0,T];L^2(\Omega))}\right.\\
&+\left.\|f-f_\varepsilon\|_{H^1([0,T];L^2(\Omega))}+\|u_0-u_{0,\varepsilon}\|_{H^1_0(\Omega)}+\|u_1-u_{1,\varepsilon}\|_{H^1_0(\Omega)}\right)\\
&+\left(1+\|V\|_{L^\infty(\Omega)} \right)\|u_0-u_\varepsilon\|_{H^2_0(\Omega)}.
\end{align*}
Since
$$\|u_{0}-{u}_{0,\varepsilon}\|_{H^2_0(\Omega)}\to 0,\quad \|u_{1}-{u}_{1,\varepsilon}\|_{H^1_0(\Omega)}\to 0,\quad \|V-V_\varepsilon\|_{L^\infty(\Omega)}\to 0,$$
\[
\|f-f_\varepsilon\|_{H^1([0,T];L^2(\Omega))}\to0
\]
for $\varepsilon\to 0$ and $u$ is a very weak solution of the initial/boundary problem \eqref{5.2}, we get
$$\|\Tilde{U}_\varepsilon\|_{M(\Omega)}\to 0$$
for $\varepsilon\to 0$. This proves Theorem \ref{thm5.3}.
\end{proof}


\begin{thebibliography}{99}
\bibitem[Af-Roz]{Af-Roz} K.~Afri, A~Rozanova-Pierrat, \newblock Dirichlet-to-Neumann or Poincare-Steklov operator on fractals described by d-sets, \newblock {\it Discrete and Continious Dynamical Systems}, S 12 (2019).

\bibitem[ART19]{ART19} A. Altybay, M. Ruzhansky, N. Tokmagambetov. \newblock Wave equation with distributional propagation speed and mass term: Numerical simulations. \newblock {\it Appl. Math. E-Notes}, 19 (2019), 552-562.

\bibitem[ARST21a]{ARST21a} A. Altybay, M. Ruzhansky, M. E. Sebih, N. Tokmagambetov. \newblock Fractional Klein-Gordon equation with singular mass. \newblock {\it Chaos, Solitons and Fractals}, 143 (2021) 110579.

\bibitem[ARST21b]{ARST21b} A. Altybay, M. Ruzhansky, M. E. Sebih, N. Tokmagambetov. \newblock Fractional Schr\"{o}dinger Equations with potentials of higher-order singularities. \newblock {\it Rep. Math. Phys.}, 87 (1) (2021) 129-144.

\bibitem[ARST21c]{ARST21c} A. Altybay, M. Ruzhansky, M. E. Sebih, N. Tokmagambetov. \newblock The heat equation with strongly singular potentials. \newblock {\it Applied Mathematics and Computation}, 399 (2021) 126006.

\bibitem[CDRT23]{CDRT23} M. Chatzakou,  A. Dasgupta, M. Ruzhansky, A. Tushir. \newblock Discrete heat equation with irregular thermal conductivity and tempered distributional data. \newblock {\it Proc. Roy. Soc. of Edinburgh Section A: Mathematics}, 1-24 (2023).

\bibitem[CRT21]{CRT21} M. Chatzakou, M. Ruzhansky, N. Tokmagambetov. \newblock Fractional Klein-Gordon equation with singular mass. II: Hypoelliptic case. \newblock {\it Complex Var. Elliptic Equ.} 67:3, 615-632 (2021).

\bibitem[CRT22a]{CRT22a} M. Chatzakou, M. Ruzhansky, N. Tokmagambetov. \newblock The heat equation with singular potentials. II: Hypoelliptic case. \newblock {\it Acta Appl. Math.}, (2022), 179:2.

\bibitem[CRT22b]{CRT22b} M. Chatzakou, M. Ruzhansky, N. Tokmagambetov. \newblock Fractional Schrödinger equations with singular potentials of higher order. II: Hypoelliptic case. \newblock {\it Rep. Math. Phys.}, 89 (2022), 59-79.

\bibitem[DRP-22]{DRP-22} 
A.~Dekkers, A.~Rozanova-Pierrat, \newblock Dirichlet boundary valued problems for linear and nonlinear wave equations on arbitrary and fractal domains, \newblock{\it JMAA}, 512, 1 (2022), 126089, ISSN 0022-247X, https://doi.org/10.1016/j.jmaa.2022.126089.

\bibitem[DHP07]{DHP07} R.~Denk, M.~Hieber, J.~Prüss, \newblock Optimal $L^p$-$L^q$-estimates for parabolic
boundary value problems with inhomogeneous data, \newblock {\it Math. Z.} 257 (1) (2007), 193-224. doi:10.1007/s00209-007-0120-9.

\bibitem[ER18]{ER18} M. R. Ebert, M. Reissig. \newblock Methods for Partial Differential Equations. \newblock {\it Birkhäuser}, 2018.

\bibitem[Evans]{Evans} L.C.~Evans, \newblock Partial differential equations, \newblock{\it American Math Society}, 2010.

\bibitem[FJ-98]{Friedlander} F.G.~Friedlander, M.~Joshi. \newblock Introduction to the Theory of Distributions, \newblock{\it Cambridge University Press}, 1998.

\bibitem[Gar21]{Gar21} C. Garetto. \newblock On the wave equation with multiplicities and space-dependent irregular coefficients.
\newblock {\it Trans. Amer. Math. Soc.}, 374 (2021), 3131-3176.

\bibitem[GR15]{GR15} C. Garetto, M. Ruzhansky. \newblock Hyperbolic second order equations with non-regular time dependent coefficients.
\newblock {\it Arch. Rational Mech. Anal.}, 217 (2015), no. 1, 113--154.

\bibitem[GS24]{GS24}C. Garetto, B. Sabitbek. \newblock Hyperbolic systems with non-diagonalisable principal part and variable multiplicities, III: singular coefficients. \newblock{\it Math. Ann.} 390, 1583–1613 (2024). https://doi.org/10.1007/s00208-023-02792-7



\bibitem[Gris]{Gris} P.~Grisvard, \newblock Elliptic problems in nonsmooth domains, \newblock{\it Vol. 69 of Classics in Applied Mathematics, Society for Industrial and Applied Mathematics (SIAM)}, Philadelphia, PA, 2011, reprint of the 1985
original [ MR0775683], With a foreword by Susanne C. Brenner. doi:10.1137/1.9781611972030.ch1.

\bibitem[JW-84]{JW-84} A.~Jonsson and H.~Wallin, \newblock{Function spaces on subsets of $\mathbb{R}^n$}, \newblock{\it Math. Reports 2}, Part 1, Harwood Acad. Publ. London, 1984.

\bibitem[KL09]{KL09} B.~Kaltenbacher, I.~Lasiecka, \newblock{Global existence and exponential decay rates for the Westervelt equation}, \newblock{\it Discrete Contin. Dyn. Syst. Ser.} S, 2(3) (2009), 503-523. doi:10.3934/dcdss.2009.2.503.

\bibitem[KL11]{KL11} B.~Kaltenbacher, I.~Lasiecka, \newblock{Well-posedness of the Westervelt and the Kuznetsov equation with nonhomogeneous Neumann boundary conditions}, \newblock{\it Conference Publications} (2011), 763-773. doi:10.3934/proc.2011.2011.763 

\bibitem[MeWi]{MeWi} S.~Meyer, M.~Wilke. \newblock{Global well-posedness and exponential stability for Kuznetsov's equation in 
$L_p$-spaces}. \newblock{\it Evol. Equ. Control Theory}, 2 (2) (2013): 365-378. doi:10.3934/eect.2013.2.365

\bibitem[MRT19]{MRT19} J. C. Munoz, M. Ruzhansky, N. Tokmagambetov, \newblock Wave propagation with irregular dissipation and applications to acoustic problems and shallow water. \newblock {\it Journal de Math\'ematiques Pures et Appliqu\'ees}. 123 (2019), 127-147.

\bibitem[Nyst]{Nyst} K. Nyström. \newblock Integrability of Green potentials in fractal domains. \newblock{\it Ark. Mat.} 34 (2) (1996): 335-381. doi:10.1007/BF02559551.

\bibitem[RSY22]{RSY22} M. Ruzhansky, S. Shaimardan, A. Yeskermessuly. \newblock Wave equation for Sturm-Liouville operator with singular potentials. \newblock {\it J. Math. Anal. Appl.}, 531, 1, 2, (2024), 127783.

\bibitem[RT17a]{RT17a} M. Ruzhansky, N. Tokmagambetov. \newblock Very weak solutions of wave equation for Landau Hamiltonian with irregular electromagnetic field. \newblock {\it Lett. Math. Phys.}, 107 (2017) 591-618.

\bibitem[RT17b]{RT17b} M. Ruzhansky, N. Tokmagambetov. \newblock Wave equation for operators with discrete spectrum and irregular propagation speed. \newblock {\it Arch. Rational Mech. Anal.}, 226 (3) (2017) 1161-1207.

\bibitem[RY22]{RY22} M. Ruzhansky, A. Yeskermessuly. \newblock Wave equation for Sturm-Liouville operator with singular intermediate coefficient and potential. \newblock {\it Bull. Malays. Math. Sci. Soc.}, 46, 195 (2023).

\bibitem[RY24a]{RY24a} M. Ruzhansky, A. Yeskermessuly. \newblock Heat equation for Sturm-Liouville operator with singular propagation and potential. \newblock {\it J. Appl. Anal.}, (2024). https://doi.org/10.1515/jaa-2023-0146

\bibitem[RY24b]{RY24b} M.~Ruzhansky, A.~Yeskermessuly, \newblock Schrödinger equation for Sturm–Liouville operator with singular propagation and potential. \newblock{\it Z. Anal. Anwend.} (2024), published online first, DOI 10.4171/ZAA/1756

\bibitem[Sch54]{Sch54} L. Schwartz. \newblock Sur l’impossibilité de la multiplication des distributions. \newblock {\it C. R. Acad. Sci. Paris}, 239 (1954) 847–848.

\bibitem[SW22]{SW22} M.E. Sebih, J. Wirth. \newblock On a wave equation with singular dissipation. \newblock {\it Math. Nachr.}, 295 (2022), 1591–1616.

\bibitem[Tesch]{Tesch} G.~Teschl. \newblock Ordinary Differential Equations and Dynamical Systems. Graduate Studies in Mathematics 140. \newblock{\it American Mathematical Soc.} (2012).

\bibitem[Xie]{Xie} W. Xie. \newblock A sharp pointwise bound for functions with $L^2$-Laplacians and zero boundary values of arbitrary three-dimensional domains, \newblock{\it Indiana Univ. Math. J.} 40 (4) (1991): 1185-1192. doi:10.1512/iumj.1991.40.40052.











\end{thebibliography}
\end{document}